\documentclass[12pt,a4paper,twoside,reqno]{amsart}
\usepackage[british,UKenglish,USenglish,english,american]{babel}

\usepackage[utf8]{inputenc}

\usepackage{multirow}
\usepackage{multicol}

\usepackage{amsfonts, amsmath, amsthm, amssymb, mathrsfs, amscd, mathtools}
\usepackage{enumerate}
\usepackage[all]{xy}
\usepackage{faktor}

\usepackage{url, enumitem}
\usepackage[noadjust]{cite}

\numberwithin{equation}{section}
\allowdisplaybreaks

\usepackage{fancyhdr,lipsum}

\usepackage[pdftex,%
bookmarks=true,
breaklinks=true,
bookmarksnumbered = true,
colorlinks= true,
urlcolor= blue,
anchorcolor = yellow,
citecolor=blue,
]{hyperref}      

\newtheoremstyle{thm}{}{}{\itshape}{}{\scshape}{.}{ }{}
\theoremstyle{thm}

\newtheorem{theorem}{Theorem}[section]
\newtheorem{proposition}[theorem]{Proposition}
\newtheorem{lemma}[theorem]{Lemma}
\newtheorem{corollary}[theorem]{Corollary}

\newtheoremstyle{remark}{}{}{}{}{\scshape}{.}{ }{}
\theoremstyle{remark}

\newtheorem{definition}[theorem]{Definition}
\newtheorem{remark}[theorem]{Remark}

\begin{document}

\title[Forbidden relations in universal virtual braid groups]{Forbidden relations in universal virtual braid groups}

\author[O.~Ocampo]{Oscar Ocampo}
\address{Oscar Ocampo\\ Universidade Federal da Bahia, Departamento de Matem\'atica - IME, CEP:~40170-110 - Salvador, Brazil}
\email{oscaro@ufba.br}

\author[H.~Stylianakis]{Charalampos Stylianakis}
\address{Charalampos Stylianakis\\ University of the Aegean, Department of mathematics, Karlovassi, 83200, Samos, Greece}
\email{stylianakisy2009@gmail.com}



\date{\today}

\begin{abstract}
We study natural automorphisms of the universal virtual braid group $UV_n(k)$. These automorphisms induce commuting involutions in the outer automorphism group and generate a subgroup isomorphic to $\mathbb{Z}_2^k\times\mathbb{Z}_2$. We then show that the two one-forbidden quotients of $UV_n(k)$ are isomorphic.  
Furthermore, we introduce the universal unrestricted virtual braid group $UUV_n(k)$ obtained by imposing simultaneously the two forbidden relations, and derive several structural properties inherited from the universal setting. 
Since the multi-virtual braid group $M_kVB_n$ is a quotient of $UV_n(k)$, the corresponding results for $M_kVB_n$ follow as consequences. In particular, for $k=1$ we prove that the quotients of $VB_n$ by the two forbidden relations are isomorphic and obtain structural properties for the unrestricted virtual braid group.
\end{abstract}

\maketitle

\section{Introduction}

The virtual braid group $VB_n$, introduced by Kauffman \cite{K} and further studied by several authors, is a natural extension of the classical braid group obtained by adding virtual crossings. Over the years, several quotients and generalizations of $VB_n$ have appeared in the literature, including welded braid groups, unrestricted virtual braid groups, and more recently the multi-virtual braid groups $M_kVB_n$ \cite{BKNP, K2}.  
Another natural generalization is given by the universal virtual braid groups $UV_n(k)$ introduced in \cite{Ocampo}. These groups provide a common framework containing several braid-type groups as quotients. In particular, Proposition~2.2 of \cite{Ocampo} shows that $VB_n$ and $M_kVB_n$ are quotients of $UV_n(k)$. 

The outer automorphism group of the virtual braid group was determined by Bellingeri and Paris \cite{BP}, who proved that $\operatorname{Out}(VB_n)\cong \mathbb Z_2\times\mathbb Z_2$ for $n\ge 5$. In contrast, we show that the universal virtual braid groups admit a natural elementary Abelian $2$-subgroup $(\mathbb Z_2)^{k+1} \leq \operatorname{Out}(UV_n(k))$. 
Thus, the universal setting gives rise to increasingly large elementary Abelian $2$-subgroups in the outer automorphism group as the number of crossing types increases.

One of the fundamental features of braid-type groups is the presence of the so-called forbidden relations (see for example \cite[p. 4]{BBD} and \eqref{eqn:f1t}, \eqref{eqn:f2t} in Section 3). In the classical virtual braid group setting, the two forbidden relations produce different quotients and play an important role in the theory of welded and unrestricted virtual braids; see for instance \cite{BBD}. In the universal setting, analogous forbidden relations were considered in \cite[Proposition~13]{NO}. 

A subtle point is that different conventions appear in the literature regarding the choice of forbidden relation. For example, the convention used in \cite{BKNP} differs from the one adopted in \cite{BBD}. This naturally raises the question of whether the corresponding one-forbidden quotients are genuinely different or merely different presentations of isomorphic groups.

One of the goals of this note is to study natural automorphisms of $UV_n(k)$. More particularly, for $\ell =1, \ldots, k$ we introduce two involutive types of automorphisms $\theta$ and $\zeta_\ell$, prove that they define nontrivial elements of the outer automorphism group, and show that they generate a subgroup isomorphic to $\mathbb{Z}_2^k\times\mathbb{Z}_2$ inside $\operatorname{Out}(UV_n(k))$. The corresponding statements for several quotients are then obtained by passing to quotient groups. For example, we prove that the induced automorphisms $\theta$, $\zeta_k \in \operatorname{Out}(UV_n(k))$ on the multi-virtual braid group $M_kB_n$ generate a subgroup isomorphic to $\mathbb{Z}_2\times\mathbb{Z}_2$ inside $\operatorname{Out}(M_kB_n)$.

Another goal is to show that the two one-forbidden quotients of $UV_n(k)$ are isomorphic. As a consequence, the corresponding one-forbidden quotients of $M_kVB_n$ are also isomorphic. In particular, for $k=1$, we prove that the quotients of $VB_n$ by the two forbidden relations are isomorphic. This result seems not to have been explicitly observed in the literature.

We also introduce the universal unrestricted virtual braid group
$$
\faktor{UUV_n(k)=UV_n(k)}{\langle\!\langle F1,F2\rangle\!\rangle},
$$
obtained by imposing simultaneously the two families of forbidden relations \eqref{eqn:f1t}, \eqref{eqn:f2t}. This construction should be viewed as a universal analogue of the unrestricted virtual braid group. We derive several structural properties of $UUV_n(k)$ inherited from the universal setting, including the perfectness of its commutator subgroup and the minimality of the symmetric group as its smallest non-abelian finite quotient. As consequences, analogous properties are obtained for the multi-unrestricted braid groups $M_kUB_n$ introduced in \cite{BKNP}. In particular, for $k=1$, we recover structural properties of the unrestricted virtual braid group $UVB_n$. For more details about $UVB_n$ see \cite{BBD, Makri}. 

For the reader's convenience, we summarize in Table \ref{braidlike_groups} the notation for the various braid-type groups that appear in this paper. Moreover, Diagram \ref{eqn:diagram} provides a relationship between these braid-type groups.

\begin{table}[h]
    \centering
    \begin{tabular}{ |p{2cm}|p{8cm}| }
 \hline
 Notation & Name\\
 \hline
 $UV_n(k)$ & Universal virtual braid group\\
 $UW_n(k)$ & Universal welded braid group\\
 $UUV_n(k)$ & Universal unrestricted Virtual braid group\\
 $M_kVB_n$ & Multi-virtual braid group\\
 $M_kWB_n$ & Multi-welded braid group\\
 $M_kUB_n$ & Multi-unrestricted braid group\\
 $VB_n$ & Virtual braid group\\
 $WB_n$ & Welded braid group\\
 $UVB_n$ & Unrestricted virtual braid group\\
 \hline
\end{tabular}
    \caption{Braid-type groups}
    \label{braidlike_groups}
\end{table}

The paper is organized as follows. In Section~\ref{sec:outer} we study outer automorphisms of $UV_n(k)$ and their induced automorphisms on several quotient groups. In Section~\ref{sec:forbidden} we study forbidden quotients of $UV_n(k)$ and prove that the two one-forbidden quotients are isomorphic. Finally, in Section~\ref{sec:unrestricted} we introduce the universal unrestricted virtual braid group $UUV_n(k)$, establish some of its structural properties, and derive consequences for multi-unrestricted braid groups and unrestricted virtual braid groups.

\subsection*{Acknowledgments}

O.~O.~was partially supported by the National Council for Scientific and Technological Development (CNPq, Brazil) through a \textit{Bolsa de Produtividade} grant No.~305422/2022--7.

\vspace*{0.1 cm}



\section{Outer automorphisms of $UV_n(k)$}\label{sec:outer}

Let $UV_n(k)$ denote the universal virtual braid group on $n$ strands and with $k$ types of crossings. Thus $UV_n(k)$ is generated by 
$$
\rho_i,\;\sigma_{i,t}, \qquad i=1,\dots,n-1, \quad t=1,\dots,k,
$$
with defining relations as in \cite[Definition~2.1]{Ocampo}.

\begin{theorem}\label{thm:outer-UV}
Let $n\ge 3$ and $k\ge 1$. For each $\ell=1,\dots,k$, define $\zeta_\ell\colon UV_n(k)\longrightarrow UV_n(k)$ by
$$
\zeta_\ell(\rho_i)=\rho_i, \qquad \zeta_\ell(\sigma_{i,\ell})=\rho_i\sigma_{i,\ell}\rho_i, \qquad \zeta_\ell(\sigma_{i,t})=\sigma_{i,t}\quad (t\neq \ell), 
$$
for all $i=1,\dots,n-1$. Also define $\theta\colon UV_n(k)\longrightarrow UV_n(k)$ by
$$
\theta(\rho_i)=\rho_i, \qquad \theta(\sigma_{i,t})=\sigma_{i,t}^{-1} 
$$
for all $i=1,\dots,n-1$ and $t=1,\dots,k$. 
Then $\theta,\zeta_1,\dots,\zeta_k$ induce commuting involutions in $\operatorname{Out}(UV_n(k))$ generating a subgroup isomorphic to $(\mathbb{Z}_2)^{k+1}$. 
\end{theorem}

\begin{proof}
We first prove that the maps $\theta,\zeta_1,\dots,\zeta_k$ are automorphisms.

The Coxeter relations among the generators $\rho_i$ are fixed by $\theta$.  The commutativity relations
$$
\sigma_{i,t}\sigma_{j,\ell} = \sigma_{j,\ell}\sigma_{i,t}, \qquad |i-j|\ge 2, 
$$
are preserved because commuting elements have commuting inverses. Similarly, the mixed commutativity relations
$$
\sigma_{i,t}\rho_j=\rho_j\sigma_{i,t}, \qquad |i-j|\ge 2,
$$
are preserved.  It remains to check the mixed braid relation
$$
\rho_i\rho_{i+1}\sigma_{i,t} = \sigma_{i+1,t}\rho_i\rho_{i+1}.
$$
Set 
$$
a=\rho_i,\qquad b=\rho_{i+1},\qquad x=\sigma_{i,t},\qquad y=\sigma_{i+1,t}.
$$
The relation is $abx=yab$. Since $a^2=b^2=1$, this relation is equivalent to $x=bayab$. Hence $x^{-1}=bay^{-1}ab$. Therefore $abx^{-1} = abbay^{-1}ab = y^{-1}ab$. This is precisely the image under $\theta$ of the mixed braid relation. 
Thus $\theta$ preserves all defining relations and defines an endomorphism of $UV_n(k)$. Since $\theta^2=\mathrm{id}$, the map $\theta$ is bijective, hence an automorphism.

Now fix $\ell\in\{1,\dots,k\}$. We prove that $\zeta_\ell$ is an automorphism. The relations involving only the generators $\rho_i$ are fixed by $\zeta_\ell$. 
The relations involving only the families $\sigma_{i,t}$ with $t\neq \ell$ are also fixed. For the family $\sigma_{i,\ell}$, the commutativity relations are preserved because, when $|i-j|\ge 2$, the elements
$$
\rho_i,\;\rho_j,\;\sigma_{i,\ell},\;\sigma_{j,\ell}
$$
commute in the required way. The same argument applies to the mixed commutativity relations. 
It remains to check the mixed braid relation for the distinguished family $t=\ell$. Set
$$
a=\rho_i,\qquad b=\rho_{i+1},\qquad x=\sigma_{i,\ell},\qquad y=\sigma_{i+1,\ell}.
$$
We need to prove that
$$
ab(axa)=(byb)ab.
$$
Using the Coxeter relation $aba=bab$ and the mixed braid relation $abx=yab$, we get
$$
abaxa=babxa=b(abx)a=b(yab)a=byaba.
$$
Since $aba=bab$, we have
$$
byaba=bybab=(byb)ab.
$$
Thus $\zeta_\ell$ preserves all defining relations and defines an endomorphism of $UV_n(k)$. Since $\theta^2=\mathrm{id}$, the map $\theta$ is bijective, hence an automorphism.

The automorphisms $\theta,\zeta_1,\dots,\zeta_k$ are all involutions. Moreover, they commute with each other. Indeed, $\theta$ commutes with each $\zeta_\ell$ because both maps fix every $\rho_i$, and on the generator $\sigma_{i,\ell}$ one has
$$
(\theta\circ\zeta_\ell)(\sigma_{i,\ell}) = \theta(\rho_i\sigma_{i,\ell}\rho_i) = \rho_i\sigma_{i,\ell}^{-1}\rho_i, 
$$
while
$$
(\zeta_\ell\circ\theta)(\sigma_{i,\ell}) = \zeta_\ell(\sigma_{i,\ell}^{-1}) = (\rho_i\sigma_{i,\ell}\rho_i)^{-1} = \rho_i\sigma_{i,\ell}^{-1}\rho_i. 
$$
For $t\neq \ell$, this is immediate since $\zeta_\ell$ fixes $\sigma_{i,t}$.

Similarly, if $\ell\neq m$, then $\zeta_\ell$ and $\zeta_m$ commute. They both fix all generators $\rho_i$, and each one acts nontrivially only on one family of generators $\sigma_{i,\ell}$ or $\sigma_{i,m}$, while fixing the other family.

We now prove that the classes of $\theta,\zeta_1,\dots,\zeta_k$ generate a subgroup isomorphic to $(\mathbb{Z}_2)^{k+1}$ in $\operatorname{Out}(UV_n(k))$. 
Let
$$
I\subseteq\{1,\dots,k\}
$$
and write
$$
\zeta_I=\prod_{\ell\in I}\zeta_\ell,
$$
with the convention that $\zeta_\emptyset=\mathrm{id}$. Since the $\zeta_\ell$'s commute, this product is well defined.

We first show that, if $I\neq\emptyset$, then $\zeta_I$ is not inner. Choose $\ell\in I$. Consider the natural epimorphism
$$
p_\ell\colon UV_n(k)\longrightarrow UVB_n
$$
defined by
$$
p_\ell(\rho_i)=\rho_i, \qquad p_\ell(\sigma_{i,\ell})=\sigma_i, \qquad p_\ell(\sigma_{i,t})=1 \quad (t\neq \ell).
$$
This map is well defined: it kills the generators $\sigma_{i,t}$ for $t\neq\ell$, sends $\sigma_{i,\ell}$ to $\sigma_i$, and then factors through the natural quotient $UV_n(1)\twoheadrightarrow UVB_n$ where the classical braid relations and both forbidden relations are imposed.

Let $\gamma_n\in \operatorname{Aut}(UVB_n)$ be the automorphism defined by
$$
\gamma_n(\rho_i)=\rho_i, \qquad \gamma_n(\sigma_i)=\rho_i\sigma_i\rho_i. 
$$
A direct computation on the generators gives
\begin{equation}\label{eqn:comm_hom}
p_\ell\circ \zeta_I=\gamma_n\circ p_\ell.
\end{equation}
Indeed, for the generators $\rho_i$ we have
$$
p_\ell(\zeta_I(\rho_i))=p_\ell(\rho_i)=\rho_i = \gamma_n(p_\ell(\rho_i)).
$$
For $\sigma_{i,\ell}$, since $\ell\in I$,
$$
p_\ell(\zeta_I(\sigma_{i,\ell})) = p_\ell(\rho_i\sigma_{i,\ell}\rho_i) = \rho_i\sigma_i\rho_i = \gamma_n(\sigma_i) = \gamma_n(p_\ell(\sigma_{i,\ell})).
$$
Finally, if $t\neq \ell$, then either $\zeta_I$ fixes $\sigma_{i,t}$ (if $t\notin I$) or sends it to $\rho_i\sigma_{i,t}\rho_i$ (if $t\in I$). In both cases its image under $p_\ell$ is $1$, and hence
$$
p_\ell(\zeta_I(\sigma_{i,t})) = 1 = \gamma_n(p_\ell(\sigma_{i,t})).
$$
Thus $p_\ell\circ \zeta_I=\gamma_n\circ p_\ell$.

Suppose, for contradiction, that $\zeta_I$ is inner. Then there exists $g\in UV_n(k)$ such that
$$
\zeta_I(x)=gxg^{-1}
$$
for every $x\in UV_n(k)$. Applying $p_\ell$, we obtain
$$
p_\ell(\zeta_I(x)) = p_\ell(g)p_\ell(x)p_\ell(g)^{-1}. 
$$
Using the compatibility relation $p_\ell\circ\zeta_I=\gamma_n\circ p_\ell$ given in \eqref{eqn:comm_hom}, this gives
$$
\gamma_n(p_\ell(x)) = p_\ell(g)p_\ell(x)p_\ell(g)^{-1}. 
$$
Since $p_\ell$ is surjective, it follows that $\gamma_n$ is inner in $UVB_n$. This contradicts \cite[Proposition~4]{Makri}, where $\gamma_n$ is shown to define a nontrivial element of $\operatorname{Out}(UVB_n)$. Hence $\zeta_I$ is not inner.

We now consider products involving $\theta$. In the abelianization of $UV_n(k)$, all generators $\sigma_{i,t}$ with fixed $t$ have the same image, say $s_t$, and all $\rho_i$ have the same image $r$, with $2r=0$. The mixed relations impose no further relations, so
$$
UV_n(k)^{\mathrm{ab}} \cong \mathbb{Z}^k\oplus\mathbb{Z}_2, 
$$
where the factors of $\mathbb{Z}^k$ are generated by $s_1,\dots,s_k$.

Each automorphism $\zeta_\ell$ acts trivially on the abelianization, because
$$
\zeta_\ell(\sigma_{i,\ell})=\rho_i\sigma_{i,\ell}\rho_i
$$
has the same image as $\sigma_{i,\ell}$ in the abelianization. On the other hand, $\theta$ sends
$$
s_t\longmapsto -s_t
$$
for every $t=1,\dots,k$. Therefore, for every subset $I\subseteq\{1,\dots,k\}$, the automorphism $\theta\zeta_I$ acts nontrivially on the abelianization. Since every inner automorphism acts trivially on the abelianization, $\theta\zeta_I$ is not inner.

We have shown that every nontrivial product of the classes
$$
[\theta],\;[\zeta_1],\dots,[\zeta_k]
$$
is nontrivial in $\operatorname{Out}(UV_n(k))$. Since these classes commute and all have order two, they generate an elementary Abelian $2$-group of rank $k+1$. Therefore
$$
\langle [\theta],[\zeta_1],\dots,[\zeta_k]\rangle \cong (\mathbb{Z}_2)^{k+1} \leq \operatorname{Out}(UV_n(k)).
$$
\end{proof}

\begin{remark}
The automorphism $\theta$ is not inner for every $n\ge 2$ (as seen from its action on the abelianization). The proof that the automorphisms $\zeta_1,\dots,\zeta_k$ are not inner relies on the result of Makri \cite[Proposition~4]{Makri}, which holds for $n\ge 3$.
\end{remark}

Recall from \cite[Proposition~2.2]{Ocampo} that $M_kVB_n$ is a quotient of $UV_n(k)$ under the identification
$$
\rho_i\longmapsto \rho_i^{(0)}, \qquad \sigma_{i,t}\longmapsto \rho_i^{(t)} \quad (1\le t\le k-1), \qquad \sigma_{i,k}\longmapsto \sigma_i.
$$

\begin{corollary}\label{cor:outer-MkVB}
For $n\ge 3$ and $k\ge 1$, the automorphisms $\theta,\zeta_1,\dots,\zeta_k$ descend to commuting involutive automorphisms of $M_kVB_n$. Moreover, the induced automorphisms of $\theta$ and $\zeta_k$ generate a subgroup of $\operatorname{Out}(M_kVB_n)$ isomorphic to
$\mathbb{Z}_2\times\mathbb{Z}_2$.
\end{corollary}

\begin{proof}
Recall from \cite[Proposition~2.2(i)]{Ocampo} the quotient map $q\colon UV_n(k)\longrightarrow M_kVB_n$ given by
$$
q(\rho_i)=\rho_i^{(0)},\qquad q(\sigma_{i,t})=\rho_i^{(t)}\quad (1\le t\le k-1), \qquad q(\sigma_{i,k})=\sigma_i. 
$$

We first show that the automorphisms descend to $M_kVB_n$.  The automorphism $\theta$ sends each $\sigma_{i,t}$ to $\sigma_{i,t}^{-1}$ and fixes each $\rho_i$. For $t<k$, the image $q(\sigma_{i,t})=\rho_i^{(t)}$ is an involution in $M_kVB_n$, so inverting $\sigma_{i,t}$ is compatible with the relation $q(\sigma_{i,t})^2=1$. Hence $\theta$ preserves the normal subgroup defining the quotient $M_kVB_n$ and descends to an automorphism of $M_kVB_n$.

Now fix $\ell\in\{1,\dots,k\}$. The automorphism $\zeta_\ell$ fixes all $\rho_i$ and all $\sigma_{i,t}$ with $t\neq \ell$, while
$$
\zeta_\ell(\sigma_{i,\ell})=\rho_i\sigma_{i,\ell}\rho_i.
$$
If $\ell<k$, then under the quotient map this element becomes
$$
q(\rho_i\sigma_{i,\ell}\rho_i) = \rho_i^{(0)}\rho_i^{(\ell)}\rho_i^{(0)}. 
$$
This is again an element of $M_kVB_n$ satisfying the same involutivity condition as $\rho_i^{(\ell)}$, since it is a conjugate of an involution. Therefore $\zeta_\ell$ preserves the defining normal subgroup of the quotient.

If $\ell=k$, then
$$
q(\rho_i\sigma_{i,k}\rho_i) = \rho_i^{(0)}\sigma_i\rho_i^{(0)}.
$$
Thus $\zeta_k$ also preserves the defining normal subgroup. Hence all automorphisms $\theta,\zeta_1,\dots,\zeta_k$ descend to automorphisms of $M_kVB_n$.

The induced automorphisms are involutions and commute, because the original automorphisms of $UV_n(k)$ are commuting involutions. 

For convenience, we denote the induced automorphisms of $M_kVB_n$ by $\theta$ and $\zeta_k$ again. It remains to prove the statement about $\theta$ and $\zeta_k$ in the outer automorphism group of $M_kVB_n$. The induced automorphisms $\theta$ and $\zeta_k$ are given by
\begin{align*}
\theta(\sigma_i) &=\sigma_i^{-1}, & \zeta_k(\sigma_i) &=\rho_i^{(0)}\sigma_i\rho_i^{(0)},\\
\theta(\rho_i^{(\beta)}) &=\rho_i^{(\beta)}, & \zeta_k(\rho_i^{(\beta)}) &=\rho_i^{(\beta)}.
\end{align*}

In the abelianization of $M_kVB_n$, all $\sigma_i$ have the same image $s$ of infinite order, and all $\rho_i^{(\beta)}$ have images of order two. Therefore $\theta$ sends
$$
s\longmapsto -s,
$$
and hence $\theta$ is not inner.  Similarly, $\theta\zeta_k$ also sends $s$ to $-s$, since $\zeta_k$ acts trivially on the abelianization. Hence $\theta\zeta_k$ is not inner.

Finally, $\zeta_k$ is not inner. Indeed, consider the natural epimorphism $p\colon M_kVB_n\longrightarrow UVB_n$ defined by
$$
p(\sigma_i)=\sigma_i,\qquad p(\rho_i^{(0)})=\rho_i,\qquad p(\rho_i^{(\beta)})=1\quad (1\le \beta<k), 
$$
followed by imposing the second forbidden relation, so that the target is $UVB_n$.  Under this epimorphism, the automorphism $\zeta_k$ projects to the automorphism $\gamma_n$ of $UVB_n$ defined by $\gamma_n(\sigma_i)=\rho_i\sigma_i\rho_i$, $\gamma_n(\rho_i)=\rho_i$. 
That is,
$$
p\circ\zeta_k=\gamma_n\circ p.
$$
If $\zeta_k$ were inner in $M_kVB_n$, then $\gamma_n$ would be inner in $UVB_n$, because $p$ is surjective. This contradicts \cite[Proposition~4]{Makri}. Hence $\zeta_k$ is not inner.

Thus the three automorphisms $\theta$, $\zeta_k$, $\theta\zeta_k$ are non-inner. Since $\theta$ and $\zeta_k$ commute and have order two, their classes generate a subgroup of $\operatorname{Out}(M_kVB_n)$ isomorphic to $\mathbb{Z}_2\times\mathbb{Z}_2$. 
\end{proof}

\begin{remark}
The automorphisms $\zeta_1,\dots,\zeta_{k-1}$ of Theorem~\ref{thm:outer-UV} also descend to involutive automorphisms of $M_kVB_n$, and they commute with $\theta$ and $\zeta_k$. 
However, it is not known whether they define inner or outer automorphisms of $M_kVB_n$. The projection argument used for $\zeta_k$ no longer applies, because the generators corresponding to the families $\sigma_{i,\ell}$ $(\ell<k)$ become involutions in the quotient $M_kVB_n$, and therefore do not project naturally onto the braid generators of $UVB_n$. 
Determining the full subgroup of $\operatorname{Out}(M_kVB_n)$ generated by $\theta,\zeta_1,\dots,\zeta_k$ remains an open problem.
\end{remark}


\section{Forbidden relations}\label{sec:forbidden}

We now consider the two forbidden relations in the universal virtual braid group. For $1\le t\le k$ and $i=1,\dots,n-2$, define 
\begin{align}
(F1_t)\qquad \rho_i\sigma_{i+1,t}\sigma_{i,t} &= \sigma_{i+1,t}\sigma_{i,t}\rho_{i+1},\label{eqn:f1t}\\
(F2_t)\qquad \rho_{i+1}\sigma_{i,t}\sigma_{i+1,t} &= \sigma_{i,t}\sigma_{i+1,t}\rho_i.\label{eqn:f2t}
\end{align}
These are precisely the two forbidden relations considered in the universal setting; see \cite[Proposition~13]{NO}.


Let $(F1)$ denote the family of all relations $(F1_t)$ given in \eqref{eqn:f1t}, and let $(F2)$ denote the family of all relations $(F2_t)$ given in \eqref{eqn:f2t}, for $1\le t\le k$ and $i=1,\dots,n-2$.

\begin{theorem}\label{thm:forbidden-UV}
For every $k\ge 1$ and $n\ge 3$, the one-forbidden quotients
$$
Q_1 = \faktor{UV_n(k)}{\langle\!\langle F1\rangle\!\rangle} \qquad \text{and} \qquad Q_2 = \faktor{UV_n(k)}{\langle\!\langle F2\rangle\!\rangle}
$$
are isomorphic.
\end{theorem}

\begin{proof}
Consider the automorphism $\theta$ from Theorem~\ref{thm:outer-UV}. Applying $\theta$ to relation $(F1_t)$ gives
$$
\rho_i\sigma_{i+1,t}^{-1}\sigma_{i,t}^{-1} = \sigma_{i+1,t}^{-1}\sigma_{i,t}^{-1}\rho_{i+1}. 
$$
Since $\sigma_{i+1,t}^{-1}\sigma_{i,t}^{-1} = (\sigma_{i,t}\sigma_{i+1,t})^{-1}$, this can be rewritten as $\rho_i(\sigma_{i,t}\sigma_{i+1,t})^{-1} = (\sigma_{i,t}\sigma_{i+1,t})^{-1}\rho_{i+1}$ and so, we obtain $(\sigma_{i,t}\sigma_{i+1,t})\rho_i = \rho_{i+1}(\sigma_{i,t}\sigma_{i+1,t})$. 
Equivalently,
$$
\rho_{i+1}\sigma_{i,t}\sigma_{i+1,t} = \sigma_{i,t}\sigma_{i+1,t}\rho_i, 
$$
which is precisely relation $(F2_t)$.

Therefore
$$
\theta\bigl(\langle\!\langle F1\rangle\!\rangle\bigr) = \langle\!\langle F2\rangle\!\rangle.
$$
Hence $\theta$ induces an isomorphism
$$
\overline{\theta}\colon \faktor{UV_n(k)}{\langle\!\langle F1\rangle\!\rangle} \longrightarrow \faktor{UV_n(k)}{\langle\!\langle F2\rangle\!\rangle}
$$
such that the following diagram commutes:
$$
\xymatrix{
UV_n(k) \ar[r]^{\theta} \ar[d] & UV_n(k) \ar[d] \\
\faktor{UV_n(k)}{\langle\!\langle F1\rangle\!\rangle} \ar[r]^{\overline{\theta}}_{\cong} & \faktor{UV_n(k)}{\langle\!\langle F2\rangle\!\rangle}.
}
$$
\end{proof}

Theorem~\ref{thm:forbidden-UV} shows that the notion of a one-forbidden quotient is independent, up to isomorphism, of the chosen forbidden relation already at the universal level $UV_n(k)$. The corresponding statements for $M_kVB_n$, $VB_n$, and $WB_n$ are obtained by passing to quotients.

\begin{corollary}\label{cor:forbidden-MkVB}
For every $k\ge 1$ and $n\ge 3$, the one-forbidden quotients of $M_kVB_n$ associated to the two forbidden relations are isomorphic. 
\end{corollary}

\begin{proof}
The group $M_kVB_n$ is obtained from $UV_n(k)$ by imposing the additional relations described in \cite[Proposition~2.2]{Ocampo}. Under the quotient map $UV_n(k)\longrightarrow M_kVB_n$, we have
$$
\rho_i\longmapsto \rho_i^{(0)}, \qquad \sigma_{i,t}\longmapsto \rho_i^{(t)} \quad (1\le t<k), \qquad \sigma_{i,k}\longmapsto \sigma_i.
$$
The automorphism $\theta$ preserves the normal subgroup defining $M_kVB_n$ and sends the normal closure of $(F1)$ onto the normal closure of $(F2)$ by Theorem~\ref{thm:forbidden-UV}. Therefore it descends to an isomorphism between the corresponding one-forbidden quotients of $M_kVB_n$.
\end{proof}

\begin{remark}
For $k=1$, we have $M_1VB_n=VB_n$. Corollary~\ref{cor:forbidden-MkVB} therefore gives
$$
\faktor{VB_n}{\langle\!\langle F1\rangle\!\rangle} \cong \faktor{VB_n}{\langle\!\langle F2\rangle\!\rangle}. 
$$
The quotient $\faktor{VB_n}{\langle\!\langle F1\rangle\!\rangle}$ is equal to the welded braid group $WB_n$; see \cite[Definition~2.1]{BBD}.  
Hence the quotient of $VB_n$ by adding the other forbidden relation (F2) is also abstractly isomorphic to $WB_n$.

This isomorphism is not obvious from the standard presentation of $WB_n$, since the second forbidden relation does not hold in $WB_n$ when expressed in the original generators \cite[Remark~2.2]{BBD}. The isomorphism is induced by the automorphism
$$
\theta(\sigma_i)=\sigma_i^{-1}, \qquad \theta(\rho_i)=\rho_i,
$$
which sends the normal closure of one forbidden relation onto the normal closure of the other.
\end{remark}

\section{The universal unrestricted virtual braid group}\label{sec:unrestricted}

We now impose simultaneously the two families of forbidden relations. This construction should be viewed as a universal analogue of the unrestricted virtual braid group studied in \cite{BBD,Makri}.

\begin{definition}
Let $n\ge 3$ and $k\ge 1$. The \emph{universal unrestricted virtual braid group} is the quotient
$$
UUV_n(k) = \faktor{UV_n(k)} {\langle\!\langle F1,F2\rangle\!\rangle},
$$
where $(F1)$ and $(F2)$ denote the families of relations $(F1_t)$ and $(F2_t)$ given in \eqref{eqn:f1t} and \eqref{eqn:f2t}, for all $1\le t\le k$ and $i=1,\dots,n-2$.
\end{definition}

\begin{remark}
For $k=1$, the unrestricted virtual braid group $UVB_n$ of \cite{BBD} is obtained as a quotient of $UUV_n(1)$ by imposing the classical braid relations
$$
\sigma_{i,1}\sigma_{i+1,1}\sigma_{i,1} = \sigma_{i+1,1}\sigma_{i,1}\sigma_{i+1,1}.
$$
Thus $UUV_n(1)$ should be regarded as a universal unrestricted virtual braid group.
\end{remark}

There is a natural epimorphism $\pi\colon UUV_n(k)\longrightarrow S_n$ defined by
$$
\pi(\rho_i)=s_i, \qquad \pi(\sigma_{i,t})=s_i,
$$
where $s_i=(i,i+1)$. We define the \emph{pure universal unrestricted virtual braid group} by $PUUV_n(k)=\ker(\pi)$. 
For later use, define
$$
\lambda_{i,i+1}^{(t)}=\rho_i\sigma_{i,t}^{-1}, \qquad \lambda_{i+1,i}^{(t)}=\sigma_{i,t}^{-1}\rho_i,
$$
and, for $1\le i<j-1\le n-1$,
$$
\lambda_{i,j}^{(t)} = \rho_{j-1}\rho_{j-2}\cdots\rho_{i+1} \lambda_{i,i+1}^{(t)} \rho_{i+1}\cdots\rho_{j-2}\rho_{j-1},
$$
$$
\lambda_{j,i}^{(t)} = \rho_{j-1}\rho_{j-2}\cdots\rho_{i+1} \lambda_{i+1,i}^{(t)} \rho_{i+1}\cdots\rho_{j-2}\rho_{j-1}.
$$

\begin{lemma}\label{lem:UUV-conjugation}
For every $s\in S_n$, every $1\le i\ne j\le n$ and every $1\le t\le k$, one has
$$
\iota(s)\lambda_{i,j}^{(t)}\iota(s)^{-1} = \lambda_{s(i),s(j)}^{(t)},
$$
where $\iota\colon S_n\to UUV_n(k)$ is the section defined by $\iota(s_i)=\rho_i$.
\end{lemma}

\begin{proof}
It is enough to prove the statement for the standard transpositions $s_m=(m,m+1)$. 
Since $\iota(s_m)=\rho_m$, the result follows from the Coxeter relations among the generators $\rho_i$ and from the definitions of the elements $\lambda_{i,j}^{(t)}$.
\end{proof}

\begin{corollary}
The symmetric group $S_n$ acts by conjugation on the set
$$
\{\lambda_{i,j}^{(t)} \mid 1\le i\ne j\le n,\; 1\le t\le k \}.
$$
For each fixed $t$, this action is transitive on $\{\lambda_{i,j}^{(t)}\mid 1\le i\ne j\le n \}$. 
\end{corollary}

\begin{proof}
By Lemma~\ref{lem:UUV-conjugation}, for every $s\in S_n$ one has $\iota(s)\lambda_{i,j}^{(t)}\iota(s)^{-1} = \lambda_{s(i),s(j)}^{(t)}$.  
Since $s$ is a permutation of $\{1,\dots,n\}$, if $i\ne j$, then $s(i)\ne s(j)$. Therefore conjugation by $\iota(s)$ sends each element of the set
$$
\{\lambda_{i,j}^{(t)} \mid 1\le i\ne j\le n,\; 1\le t\le k \}
$$
to another element of the same set. Hence $S_n$ acts by conjugation on this set.

Now fix $t$. Given two elements
$$
\lambda_{i,j}^{(t)} \quad\text{and}\quad \lambda_{r,s}^{(t)}
$$
with $i\ne j$ and $r\ne s$, there exists a permutation $\tau\in S_n$ such that
$$
\tau(i)=r \qquad\text{and}\qquad \tau(j)=s.
$$
Applying Lemma~\ref{lem:UUV-conjugation} again, we get
$$
\iota(\tau)\lambda_{i,j}^{(t)}\iota(\tau)^{-1} = \lambda_{\tau(i),\tau(j)}^{(t)} = \lambda_{r,s}^{(t)}.
$$
Thus all elements $\lambda_{i,j}^{(t)}$ with fixed $t$ lie in the same orbit. Therefore the action is transitive on $\{\lambda_{i,j}^{(t)}\mid 1\le i\ne j\le n\}$. 
\end{proof}

\begin{proposition}\label{prop:UUV-semidirect}
The group $UUV_n(k)$ splits as
$$
UUV_n(k)\cong PUUV_n(k)\rtimes S_n.
$$
\end{proposition}

\begin{proof}
The epimorphism $\pi\colon UUV_n(k)\longrightarrow S_n$ admits the natural section $\iota\colon S_n\longrightarrow UUV_n(k)$, defined by $\iota(s_i)=\rho_i$. 
Therefore
$$
UUV_n(k) \cong \ker(\pi)\rtimes S_n = PUUV_n(k)\rtimes S_n. 
$$
\end{proof}

\begin{proposition}\label{prop:UUV-ab}
For every $n\ge 3$ and $k\ge 1$, the abelianization of $UUV_n(k)$ is 
$$
UUV_n(k)^{\mathrm{ab}} \cong \mathbb{Z}^k\oplus\mathbb{Z}_2. 
$$
\end{proposition}

\begin{proof}
By definition, $UUV_n(k)$ has the presentation obtained from the presentation of $UV_n(k)$ by adding the relations $(F1_t)$ and $(F2_t)$. 

In the abelianization all generators commute. From the Coxeter relations among the $\rho_i$, all $\rho_i$ have the same image $r$, and $2r=0$. From the mixed relation $\rho_i\rho_{i+1}\sigma_{i,t} = \sigma_{i+1,t}\rho_i\rho_{i+1}$, we obtain
$$
\sigma_{i,t}=\sigma_{i+1,t}
$$
in the abelianization. Hence, for each fixed $t$, all $\sigma_{i,t}$ have the same image, which we denote by $s_t$. 
Notice that the forbidden relations impose no additional relation in the abelianization. 

Therefore the abelianization is generated by $s_1,\dots,s_k,r$, with the only additive relation $2r=0$. Hence
$$
UUV_n(k)^{\mathrm{ab}} \cong \mathbb{Z}^k\oplus\mathbb{Z}_2. 
$$
\end{proof}

The next proposition follows from the corresponding results for universal virtual braid groups established in \cite[Proposition~5.1 and Theorem~5.5]{Ocampo}.

\begin{proposition}
Let $n\ge 5$ and $k\ge 1$.
\begin{enumerate}
    \item The commutator subgroup $UUV_n(k)'$ is perfect. 
    \item The symmetric group $S_n$ is the smallest non-abelian finite quotient of $UUV_n(k)$.
\end{enumerate}

\end{proposition}

\begin{proof} 
Let $n\ge 5$ and $k\ge 1$.
\begin{enumerate}
    \item The quotient map $\pi\colon UV_n(k)\to UUV_n(k)$ has kernel normally generated by the forbidden relations $F1_t$ and $F2_t$. A direct computation shows that each $F1_t$ and $F2_t$ lies in $UV_n(k)'$; indeed, $F1_t = [\rho_{i+1},(\sigma_{i+1,t}\sigma_{i,t})^{-1}] \cdot (\rho_i\rho_{i+1}^{-1})$ and $\rho_i\rho_{i+1}^{-1}\in UV_n(k)'$ because $\rho_i$ and $\rho_{i+1}$ are conjugate. Hence $\ker\pi\subseteq UV_n(k)'$. 

Consequently, $\pi$ induces a surjective homomorphism $UV_n(k)'\to UUV_n(k)'$. By \cite[Proposition~5.1]{Ocampo}, $UV_n(k)'$ is perfect for $n\ge5$. Since homomorphic images of perfect groups are perfect, $UUV_n(k)'$ is perfect.

    \item This follows immediately from \cite[Theorem~5.5]{Ocampo}, since $UUV_n(k)$ is a quotient of $UV_n(k)$ preserving the virtual generators.
\end{enumerate}
\end{proof}

\begin{remark}
The structure of the pure subgroup $PUUV_n(k)$ appears to be considerably more subtle. Inspired by the unrestricted virtual braid group case studied in \cite{BBD,Makri}, it is natural to expect that $PUUV_n(k)$ admits a description in terms of the generators
$$
\lambda_{i,j}^{(t)}
$$
and commutation relations determined by unordered pairs of indices. 
However, such a description would require a detailed Reidemeister--Schreier analysis of the pure universal virtual braid group $PUV_n(k)$, which lies beyond the scope of the present note.
\end{remark}

The multi-unrestricted braid group $M_kUB_n$ introduced in \cite[Section~4]{BKNP} is naturally obtained as a quotient of $UUV_n(k)$. 
Indeed, the quotient map $UV_n(k)\longrightarrow M_kVB_n$ factors through the unrestricted quotient by imposing simultaneously the two forbidden relations. 
In \eqref{eqn:diagram} we illustrate a chain of natural epimorhisms (canonical quotients) among universal virtual braid groups, multi-virtual braid groups (with $k\geq 1$ virtual crossings) and their respective welded and unrestricted versions. 
\begin{equation}\label{eqn:diagram}
    \xymatrix@C=3.5pc{
 UV_n(k) \ar@{->>}[d] \ar@{->>}[rd] \ar@{->>}[rrr] & & & UV_n(1) \ar@{->>}[ld] \ar@{->>}[d] \\
M_kVB_n \ar@{->>}[d] & UW_n(k) \ar@{->>}[d] \ar@{->>}[ld] \ar@{->>}[r] & UW_n(1) \ar@{->>}[d] \ar@{->>}[rd]  & VB_n \ar@{->>}[d] \\
M_kWB_n \ar@{->>}[rd] & UUV_n(k) \ar@{->>}[d] \ar@{->>}[r] & UUV_n(1) \ar@{->>}[d] & WB_n \ar@{->>}[ld]\\
& M_kUB_n \ar@{->>}[r] & UVB_n  & 
}
\end{equation} 

\begin{corollary}\label{cor:MkUB-properties}
Let $n\ge 5$ and $k\ge 1$. The commutator subgroup of $M_kUB_n$ is perfect. Moreover, $S_n$ is the smallest non-abelian finite quotient of $M_kUB_n$.
\end{corollary}

\begin{proof}
Since $M_kVB_n$ is a quotient of $UV_n(k)$, it follows that $M_kUB_n$ is a quotient of  $UV_n(k)$. By \cite[Proposition~5.1]{Ocampo}, the commutator subgroup of every quotient of $UV_n(k)$ is perfect for $n\ge 5$. Therefore $(M_kUB_n)'$ is perfect.

Moreover, the quotient defining $M_kUB_n$ preserves the virtual generators $\rho_i^{(0)}$, whose images still generate a copy of $S_n$. Hence, by \cite[Proposition~5.6]{Ocampo}, the smallest non-abelian finite quotient of $M_kUB_n$ is $S_n$. 
\end{proof}

For $k=1$, Corollary~\ref{cor:MkUB-properties} applies to the unrestricted virtual braid group $UVB_n$. In particular, for $n\ge 5$, the commutator subgroup $UVB_n'$ is perfect and $S_n$ is the smallest non-abelian finite quotient of $UVB_n$. 
To the best of our knowledge, these properties do not seem to have been explicitly observed in the literature on unrestricted virtual braid groups.

\begin{remark}
The group $UUV_n(k)$ provides a universal unrestricted object for multi-virtual braid-type groups. In particular, the rigidity properties inherited from $UV_n(k)$ descend not only to $UUV_n(k)$ but also to the multi-unrestricted braid group $M_kUB_n$ introduced in \cite[Section~4]{BKNP}.
\end{remark}

We conclude this section with the following result about outer automorphisms of some quotients of the universal virtual braid group. 

\begin{proposition}\label{prop:outer-quotient}
Let $n\ge 3$ and $k\ge 1$, and let
$$
Q=\faktor{UV_n(k)}{N}
$$
be a quotient. Let $I\subseteq \{1,\dots,k\}$ be a nonempty subset. Assume that $N$ is invariant under $\theta$ and under $\zeta_\ell$ for every $\ell\in I$. 
Denote by $\overline{\theta}$ and $\overline{\zeta}_\ell$ the induced automorphisms of $Q$.

For every subset $J\subseteq I$, write
$$
\overline{\zeta}_J=\prod_{j\in J}\overline{\zeta}_j,
$$
with the convention that $\overline{\zeta}_\emptyset=\mathrm{id}$.

Assume that:
\begin{enumerate}
\item for every $\ell\in I$, the image of $\sigma_{1,\ell}$ has infinite order in $Q^{\mathrm{ab}}$;

\item for every nonempty subset $J\subseteq I$ and every $\ell\in J$, there exists an epimorphism
$$
p_\ell\colon Q\longrightarrow UVB_n
$$
such that
$$
p_\ell\circ \overline{\zeta}_J=\gamma_n\circ p_\ell, 
$$
where $\gamma_n\in \operatorname{Aut}(UVB_n)$ is defined by $\gamma_n(\sigma_i)=\rho_i\sigma_i\rho_i$, $\gamma_n(\rho_i)=\rho_i$. 
\end{enumerate}

Then the classes of $\overline{\theta}$ and $\overline{\zeta}_\ell$, for $\ell\in I$, generate a subgroup of $\operatorname{Out}(Q)$ isomorphic to $\mathbb Z_2^{|I|}\times \mathbb Z_2$. 
\end{proposition}

\begin{proof}
Since $N$ is invariant under $\theta$ and under $\zeta_\ell$ for every $\ell\in I$, these automorphisms descend to automorphisms of $Q$. Since $\theta,\zeta_1,\dots,\zeta_k$ are commuting involutions in $\operatorname{Aut}(UV_n(k))$, the induced automorphisms
$$
\overline{\theta},\quad \overline{\zeta}_\ell\;(\ell\in I)
$$
are commuting involutions in $\operatorname{Aut}(Q)$.

It remains to prove that no nontrivial product of their classes is trivial in $\operatorname{Out}(Q)$. 
Let $J\subseteq I$. First suppose that $J\neq\emptyset$. We prove that $\overline{\zeta}_J$ is not inner. Choose $\ell\in J$. By assumption, there is an epimorphism $p_\ell\colon Q\longrightarrow UVB_n$ such that $p_\ell\circ\overline{\zeta}_J=\gamma_n\circ p_\ell$. 
Suppose, for contradiction, that $\overline{\zeta}_J$ is inner. Then there exists $g\in Q$ such that
$$
\overline{\zeta}_J(x)=gxg^{-1}
$$
for every $x\in Q$. Applying $p_\ell$, we obtain
$$
p_\ell(\overline{\zeta}_J(x)) = p_\ell(g)p_\ell(x)p_\ell(g)^{-1}.
$$
Using the compatibility condition, this gives
$$
\gamma_n(p_\ell(x)) = p_\ell(g)p_\ell(x)p_\ell(g)^{-1}.
$$
Since $p_\ell$ is surjective, it follows that $\gamma_n$ is inner in $UVB_n$. This contradicts \cite[Proposition~4]{Makri}. Therefore $\overline{\zeta}_J$ is not inner.

Now we prove that $\overline{\theta}\,\overline{\zeta}_J$ is not inner for every $J\subseteq I$, including $J=\emptyset$. Choose some $\ell_0\in I$, and let $\overline{s}_{\ell_0}$ be the image of $\sigma_{1,\ell_0}$ in $Q^{\mathrm{ab}}$.
By assumption, $\overline{s}_{\ell_0}$ has infinite order. 
Each automorphism $\overline{\zeta}_j$ acts trivially on $Q^{\mathrm{ab}}$. Indeed, in the abelianization, the element
$$
\overline{\zeta}_j(\sigma_{1,j}) = \rho_1\sigma_{1,j}\rho_1 
$$
has the same image as $\sigma_{1,j}$, since the image of $\rho_1$ has order two. Thus $\overline{\zeta}_J$ acts trivially on $Q^{\mathrm{ab}}$. 
On the other hand, $\overline{\theta}$ sends
$$
\overline{s}_{\ell_0}\longmapsto -\overline{s}_{\ell_0}.
$$
Since $\overline{s}_{\ell_0}$ has infinite order, this action is nontrivial. Hence $\overline{\theta}\,\overline{\zeta}_J$ acts nontrivially on $Q^{\mathrm{ab}}$, and so it cannot be inner. 

Therefore every nontrivial product of the classes
$$
[\overline{\theta}],\qquad [\overline{\zeta}_\ell]\quad(\ell\in I)
$$
is nontrivial in $\operatorname{Out}(Q)$. Since these classes commute and all have order two, they generate an elementary Abelian $2$-group of rank $|I|+1$. Hence 
$$
\left\langle
[\overline{\theta}], [\overline{\zeta}_\ell]\;(\ell\in I)
\right\rangle
\cong
\mathbb Z_2^{|I|}\times \mathbb Z_2.
$$
\end{proof}

\begin{corollary}\label{cor:outer-unrestricted}
Let $n\ge 3$ and $k\ge 1$. The automorphisms $\theta$ and $\zeta_\ell$, $\ell=1,\dots,k$, descend to automorphisms of $UUV_n(k)$. Moreover, their classes generate a subgroup of $\operatorname{Out}(UUV_n(k))$ isomorphic to $\mathbb{Z}_2^k\times\mathbb{Z}_2$. 
Furthermore, the automorphisms $\theta$ and $\zeta_k$ descend to automorphisms of $M_kUB_n$ whose classes generate a subgroup of $\operatorname{Out}(M_kUB_n)$ isomorphic to $\mathbb{Z}_2\times\mathbb{Z}_2$. 
\end{corollary}

\begin{proof}
We apply Proposition~\ref{prop:outer-quotient}. First consider
$$
Q=UUV_n(k)
=
\faktor{UV_n(k)}{\langle\!\langle F1,F2\rangle\!\rangle}.
$$
The normal subgroup $\langle\!\langle F1,F2\rangle\!\rangle$ is invariant under $\theta$ and under each $\zeta_\ell$, $\ell=1,\dots,k$. Indeed, $\theta$ interchanges the two forbidden families $(F1)$ and $(F2)$ by the proof of Theorem~\ref{thm:forbidden-UV}, while each $\zeta_\ell$ preserves their normal closure. Hence these automorphisms descend to automorphisms of $UUV_n(k)$. 

By Proposition~\ref{prop:UUV-ab}, one has
$$
UUV_n(k)^{\mathrm{ab}}\cong \mathbb{Z}^k\oplus\mathbb{Z}_2.
$$
In particular, for every $\ell=1,\dots,k$, the image of $\sigma_{1,\ell}$ has infinite order in $UUV_n(k)^{\mathrm{ab}}$.

Now let $J\subseteq\{1,\dots,k\}$ be nonempty and choose $\ell\in J$. There is a natural epimorphism
$$
p_\ell\colon UUV_n(k)\longrightarrow UVB_n
$$
obtained by sending
$$
\rho_i\longmapsto \rho_i,
\qquad
\sigma_{i,\ell}\longmapsto \sigma_i,
\qquad
\sigma_{i,t}\longmapsto 1\quad (t\neq \ell),
$$
and then imposing the classical braid relations on the remaining family $\sigma_{i,\ell}$. This epimorphism satisfies 
$$
p_\ell\circ \overline{\zeta}_J=\gamma_n\circ p_\ell,
$$
where $\gamma_n$ is the automorphism of $UVB_n$ defined by
$$
\gamma_n(\sigma_i)=\rho_i\sigma_i\rho_i, \qquad \gamma_n(\rho_i)=\rho_i.
$$
Thus all hypotheses of Proposition~\ref{prop:outer-quotient} hold for $Q=UUV_n(k)$ with $I=\{1,\dots,k\}$. Therefore the classes of $\overline{\theta},\overline{\zeta}_1,\dots,\overline{\zeta}_k$ generate a subgroup of $\operatorname{Out}(UUV_n(k))$ isomorphic to $\mathbb{Z}_2^k\times\mathbb{Z}_2$. 

Now consider $Q=M_kUB_n$. By definition, $M_kUB_n$ is obtained from $M_kVB_n$ by imposing simultaneously the two forbidden relations; equivalently, it is a quotient of $UUV_n(k)$ through the natural quotient map 
$$
UV_n(k)\longrightarrow M_kVB_n.
$$
The defining normal subgroup is invariant under $\theta$ and $\zeta_k$, since the additional relations defining $M_kVB_n$ are preserved by these automorphisms and the two forbidden families are preserved as a normal closure.

The image of $\sigma_{1,k}$ in $M_kUB_n$ is the image of the classical generator $\sigma_1$, and it has infinite order in the abelianization. Moreover, there is a natural epimorphism
$$
p\colon M_kUB_n\longrightarrow UVB_n
$$
defined by
$$
p(\sigma_i)=\sigma_i, \qquad p(\rho_i^{(0)})=\rho_i, \qquad p(\rho_i^{(\beta)})=1\quad (1\le \beta<k). 
$$
This epimorphism satisfies
$$
p\circ \overline{\zeta}_k=\gamma_n\circ p.
$$
Therefore Proposition~\ref{prop:outer-quotient} applies to $Q=M_kUB_n$ with $I=\{k\}$. Hence the classes of $\overline{\theta}$ and $\overline{\zeta}_k$ generate a subgroup of $\operatorname{Out}(M_kUB_n)$ isomorphic to $\mathbb{Z}_2\times\mathbb{Z}_2$. 
\end{proof}

\end{document}